\documentclass{amsart}

\usepackage{amssymb}
\usepackage{amsmath}
\usepackage{amsthm}
\usepackage{enumerate}
\usepackage{tikz-cd}
\usepackage[all]{xy}

\theoremstyle{plain}
\newtheorem{theorem}{Theorem}[section]

\newtheorem{proposition}[theorem]{Proposition}
\newtheorem{corollary}[theorem]{Corollary}
\newtheorem{conj}[theorem]{Conjecture}
\newtheorem{question}[theorem]{Question}

\theoremstyle{definition}
\newtheorem{definition}[theorem]{Definition}
\newtheorem{remark}[theorem]{Remark}
\newtheorem{example}[theorem]{Example}
\newtheorem*{ack}{Acknowledgements}
\newtheorem*{data}{Data Management}

\title{On the Birch--Swinnerton-Dyer conjecture and Schur indices}
\author{Matthew Bisatt and Vladimir Dokchitser}
\address{Faculty of Natural and Mathematical Sciences, Strand Campus, King's College London, London, WC2R 2LS, United Kingdom}
\email{matthew.bisatt@kcl.ac.uk}
\email{vladimir.dokchitser@kcl.ac.uk}

\input cyracc.def       
\font\tencyr=wncyr10
\def\sha{\text{\tencyr\cyracc{Sh}}}

\def\Q{{\mathbb Q}}
\def\Z{{\mathbb Z}}
\def\C{{\mathbb C}}
\def\F{{\mathbb F}}
\def\triv{{\mathbf 1}}
\def\GL{\text{GL}}
\def\cF{\mathcal{F}}
\def\cK{\mathcal{K}}
\def\cO{\mathcal{O}}

\global\long\def\qq{\mathbb{Q}}

\global\long\def\zz{\mathbb{Z}}
\global\long\def\Gal{\operatorname{Gal}}
\global\long\def\Ind{\operatorname{Ind}}
\global\long\def\Res{\operatorname{Res}}
\global\long\def\ord{\operatorname{ord}}

\global\long\def\trace{\operatorname{tr}}
\global\long\def\Aut{\operatorname{Aut}}

\begin{document}

\begin{abstract}
For every odd prime $p$, we exhibit families of irreducible Artin representations $\tau$ with the property that for every elliptic curve $E$ the order of the zero of the twisted $L$-function $L(E,\tau,s)$ at $s\!=\!1$ must be a multiple~of~$p$. Analogously, the multiplicity of $\tau$ in the Selmer group of $E$ must also be divisible by $p$. We give further examples where $\tau$ can moreover be twisted by any character that factors through the $p$-cyclotomic extension, and examples where the $L$-functions are those of twists of certain Hilbert modular forms by Dirichlet charaters. These results are conjectural, and rely on a standard generalisation of the Birch--Swinnerton-Dyer conjecture. Our main tool is the theory of Schur indices from representation theory.
\end{abstract}

\maketitle

\vskip -0.9cm \phantom{.}
\section{Making the analytic rank divisible by $p$}

There is a standard ``minimalist conjecture'' that generically the $L$-function of an elliptic curve vanishes to order 0 or 1 at $s\!=\!1$,
depending on the sign in the functional equation. As we will illustrate, this has to be used with some caution: even when the associated Galois representation is irreducible, certain $L$-functions cannot vanish to order 1 at $s\!=\!1$ --- the order of their zero should be a multiple of a (possibly large) integer $n$.

More precisely, we look at twists of elliptic curves $E$ by Artin representations~$\tau$ and their $L$-functions $L(E,\tau,s)$, that is the $L$-function associated to the tensor product of $\tau$ with the Galois representation of $E$. When $\tau$ factors through $F/\Q$ this is a factor of $L(E/F,s)$, much like the Artin $L$-function $L(\tau,s)$ is a factor of the Dedekind $\zeta$-function of~$F$.

Throughout the article $p$ and $q$ will be distinct odd primes. We write $\langle \, , \, \rangle$ for the usual inner product of characters of representations of finite groups (embedding them into $\C$ if necessary): thus $\langle X,\tau \rangle$ is the multiplicity of $\tau$ in $X$ if $\tau$ is irreducible.

\begin{theorem}
\label{mainthm}
Let $E/\Q$ be an elliptic curve. 
Let $\tau$ be an irreducible faithful Artin representation of a Galois extension $F/\Q$ with $\Gal(F/\Q) \cong C_q \rtimes C_{p^n}$ non-abelian and with $p^n \!\nmid\! q{-}1$.\\
\noindent (i) If the Birch--Swinnerton-Dyer conjecture for Artin twists (Conjecture \ref{BSDDG}) holds, then 
$$
  \ord\limits_{s=1} L(E,\tau,s) \equiv 0 \mod{p}.
$$
(ii) If the $\ell$-primary part of the Tate--Shafarevich group $\sha(E/F)[\ell^{\infty}]$ is finite, then
$$
  \langle X_{\ell}(E/F),\tau \rangle  \equiv 0 \mod{p},
$$ 
where $\ell$ is any prime and $X_{\ell}(E/F)$ is the Pontryagin dual of the $\ell^{\infty}$-Selmer group of $E/F$ tensored with $\Q_{\ell}$, viewed as a representation of $\Gal(F/\Q)$.
\end{theorem}

\noindent This result follows from Theorem \ref{mainschurthm} and Theorem \ref{schurindices}(iii). The main question we would like to raise, of course, is whether this behaviour of $L$-functions or Selmer groups can be explained without appealing to the conjectures.

It is reasonably straightforward to construct such Galois extensions $F/\Q$.

\noindent\begin{minipage}{0.68\linewidth}
Consider for simplicity the case when $C_{p^n}$ acts on $C_q$ through $C_p$. 
Such fields $F=F_{p^n}$ are constructed as the compositum of a $C_{p^n}$-extension $K_{p^n}/\Q$ and an extension $F_p/\Q$ with Galois group $C_q \rtimes C_p$ that shares a common degree $p$ subfield $K_p$ with $K_{p^n}$.
The irreducible faithful Artin representations of $\Gal(F/\Q)$ are all of the form $\tau\otimes\chi$, for any irreducible $p$-dimensional representation of $\Gal(F_p/\Q)$ and any $1$-dimensional representation $\chi$ of $\Gal(K_{p^n}/\Q)$ of order $p^n$ (see Proposition \ref{reps}).
\end{minipage}
\begin{minipage}{0.32\linewidth}
$$
\xymatrix@dr@C=1pc{
F=F_{p^n}\ar@{-}[r] \ar@{.}[d]  & K_{p^n} \ar@{-}[l] \ar@{.}[d] \\
F_p\ar@{-}[r] \ar@{-}[d]  & K_p \ar@{-}[l]_{C_q} \ar@{-}[d]^{C_p} \\ 
F_1 \ar@{-}[r] &\Q }
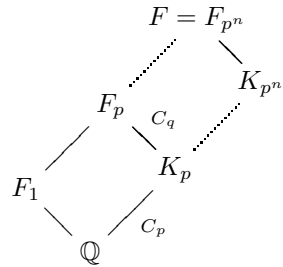
$$
\end{minipage}

For example, $K_{p^n}$ could be the $n^{\text{th}}$ layer of the $p$-cyclotomic tower of $\Q$, that is the unique degree $p^n$ subfield of $\Q(\zeta_{p^{n+1}})$, where $\zeta_{p^{n+1}}$ is a primitive $p^{n+1}$-th root of unity. This gives the following:

\begin{corollary}
Suppose that $F_p/\Q$ is Galois with $\Gal(F_p/\Q)\cong C_q \rtimes C_p$ non-abelian, and that its degree $p$ subfield $K_p$ is the first layer of the $p$-cyclotomic extension~of~$\Q$. 
Let $E/\Q$ be an elliptic curve and $\tau$ an irreducible faithful representation of~$\Gal(F_p/\Q)$. 
If Conjecture \ref{BSDDG} holds, then for all 
finite order characters $\chi$ that factor through the $p$-cyclotomic extension with $\chi^{q-1}\neq \triv$, $$ \ord_{s=1} L(E,\tau \otimes \chi,s) \equiv 0 \mod{p}. $$
\end{corollary}

If $\tau$ is a representation of $\Gal(F/\Q)$ such that $\tau = \Ind_{K/\Q} \psi$ for some subfield $K \subset F$, then we have an equality of $L$-functions $L(E,\tau,s)=L(E/K,\psi,s)$ for any elliptic curve $E/\Q$. In our setup, all irreducible faithful representations $\tau$ are induced from characters. More concretely, if $\Gal(F_{p^n}/\Q) \cong C_q \rtimes C_{p^n}$ is non-abelian, such that $C_{p^n}$ acts on $C_q$ through $C_p$, then $\tau\!=\!\Ind_{K_p/\Q} \psi$ where $K_p$ is the degree $p$ subfield of $F_{p^n}$ and $\psi$ is a primitive character of order $qp^{n-1}$.  In particular, we get the following consequence for $L$-functions of certain modular forms.

\begin{corollary}\label{hilbertcor}
Suppose that $F_p/\Q$ is Galois with $\Gal(F_p/\Q)\cong C_q \rtimes C_p$ non-abelian, and that its degree $p$ subfield $K_p$ is the first layer of the $p$-cyclotomic extension of~$\Q$.
Let $E/\Q$ be an elliptic curve, let $f_E$ be the modular form attached to $E$ and let $\mathbf{f}_E$ be the Hilbert modular form which is the base-change of $f_E$ to the (totally real cyclic) extension $K_p/\Q$. Assuming Conjecture \ref{BSDDG}, for any $n$ such that $p^n \!\nmid\! q{-}1$  and primitive character $\psi$ of $\Gal(F_pK_{p^n}/K_p)\cong C_{qp^{n-1}}$, we have 
$$
  \ord_{s=1} L(\mathbf{f}_E,\psi,s) \equiv 0 \mod{p},
$$ 
where $K_{p^n}$ is the $n^{th}$ layer of the $p$-cyclotomic extension of $\Q$.
\end{corollary}

\begin{question}
Our approach relies on elliptic curves. Are there similar phenomena for modular forms that do not correspond to elliptic curves?
\end{question}

\begin{example}
As a concrete example, take $p\!=\!3$ and $q\!=\!7$. For the degree $7$ non-Galois extension $F_1$ (see diagram above) take the field $F_1=\Q(\alpha)$ of discriminant $3^87^{12}$, where $\alpha$ is a root of $x^7{-}42x^5{-}70x^4{+}168x^3{+}126x^2{-}84x{-}45$.
As in the above discussion, take $K_{3^n}=\Q(\zeta_{3^{n+1}})^+$ and set $F_{3^n}=F_1K_{3^n}$, the $n^{\text{th}}$ layer of the $p$-cyclotomic tower of $F_1$.
The field $F_3$ is the Galois closure of $F_1$ and $\Gal(F_3/\Q) \cong C_7 \rtimes C_3$ non-abelian; this group is an analogue of a dihedral group with $C_2$ replaced by $C_3$.

The group  $\Gal(F_3/\Q) \cong C_7 \rtimes C_3$ has three $1$-dimensional representations that come from the $C_3$-quotient, and two 3-dimensional irreducible representations $\tau_0, \tau_0'$, which are induced from 1-dimensional characters $\psi_0,\psi_0'$ of $C_7$. The irreducible representations of $\Gal(F_{3^n}/\Q)\cong C_7\rtimes C_{3^n}$ are the 1-dimensional representations lifted from the $C_{3^n}$-quotient, and 3-dimensional irreducibles that can all be written as $\tau=\tau_0\otimes\chi$ or $\tau=\tau_0'\otimes\chi$ for some $1$-dimensional $\chi$; note that these can therefore also be expressed as $\tau=\Ind_{K_3/\Q}\psi$, where $\psi=\psi_0\otimes\Res\chi$ or $\psi'_0\otimes\Res\chi$ is $1$-dimensional. The faithful ones are precisely the ones with $\chi$ of maximal order, equivalently with $\psi$ of order $7\times 3^{n-1}$.

Now let $E/\Q$ be an elliptic curve. The $L$-function in Theorem \ref{mainthm} can be expressed in several ways: if, say, $\tau=\tau_0\otimes\chi=\Ind_{K_3/\Q} \psi$ is 3-dimensional irreducible, then
$$
  L(E,\tau,s) = L(E,\tau_0\otimes\chi,s) = L(E/K_3,\psi,s) = L(\mathbf{f}_E,\psi,s),
$$
where $\mathbf{f}_E$ is as in Corollary \ref{hilbertcor}.

In this setting, our prediction is that the order of vanishing of this $L$-function is necessarily a multiple of 3, so long as $\tau$ does not factor through $C_7\rtimes C_3$ (equivalently if the order of $\chi$ is at least $9$).
As we will explain in \S2--3, the corresponding statement is provably true for the Mordell--Weil group $E(F_{3^n})$, which is how we obtain the prediction for $L$-functions and Selmer groups.

Finally, let us note that it is possible to make a prediction for analytic ranks that do not involve twisted $L$-functions, although it becomes a little cumbersome.
Using the subfield lattice of $F_{3^n}/\Q$ and inductivity of $L$-functions, one checks that
$$
  \dfrac{L(E/F_{3^n},s)L(E/K_{3^{n-1}},s)}{L(E/K_{3^n},s)L(E/F_{3^{n-1}},s)} = \prod\limits_{\tau {\text { faithful}}} L(E/\Q,\tau,s)^3,
$$ 
Observe that the faithful representations $\tau\!:\!\Gal(F_{3^n}/\Q)\to\GL_3(\overline\Q)$ have Galois conjugate images, since they are induced from Galois conjugate 1-dimensional $\psi$'s. Thus, if we assume Conjecture \ref{BSDDG} or Deligne's conjecture on Galois-equivariance of $L$-values~\cite[Conjecture 2.7ii]{Deligne1979}, the orders of vanishing of their $L$-functions should all be equal, and hence the order of vanishing of the right-hand term in the above equation is a multiple of $3\times 3\times \frac{(7-1)(3^n-3^{n-1})}{3^2}=4\times 3^n$. 
In particular, if the $L$-values at $s=1$ are non-zero for $E/F_{3^{n-1}}$ and $E/K_{3^n}$ (and hence for $E/K_{3^{n-1}})$, then the order of the zero of $L(E/F_{3^n},s)$ must be a multiple of $4\times 3^n$. 
More generally, the same technique yields the following result.
\end{example}

\begin{corollary}
\label{Lfns}
Let $F/\Q$ be a Galois extension with $\Gal(F/\Q) \cong C_q \rtimes C_{p^n}$ non-abelian, where the image of $C_{p^n}$ in $\Aut C_q$ has order $p^r$ and $p^n \! \nmid q{-}1$. Suppose $E/\Q$ is an elliptic curve such that $L(E/K,1)\neq 0$ for all proper subfields $K \subsetneq F$.
If Conjecture \ref{BSDDG} holds, then 
$$
  \ord_{s=1} L(E/F,s) \, \, \equiv \, \, 0 \mod{p^{n-r}(p{-}1)(q{-}1)}.
$$
\end{corollary}

\begin{remark}
At present we do not have examples where the orders of vanishing of such $L$-functions are non-zero, as their conductors appear to be  too large for any extensive numerical search. We also cannot guarantee a zero at $s=1$ by forcing the $L$-function to be essentially antisymmetric about that point: the twisting Artin representations $\tau$ (or $\tau\otimes\chi$) above are never self-dual, so the functional equation relates $L(E,\tau)$ to $L(E,\tau^*)$ and the root number (``sign'') cannot be used to force a zero. The latter is a general feature of our approach, see Remark \ref{brauerspeiser}.
\end{remark}

\begin{remark}
As will be clear from \S2--3, Theorem \ref{mainthm} applies generally to abelian varieties over number fields, rather than elliptic curves over $\Q$.
\end{remark}

\begin{remark}
The Galois representation $H^1_{\text{\'et}}(E,\Q_{\ell})_{\C}\otimes\tau$ can be irreducible, so the multiplicity of the order of vanishing is not explained by a decomposition of the Galois representation. Moreover, the $L$-series is not the (formal) $p^{\rm{th}}$ power of another $L$-series. For example, if $G=C_7 \rtimes C_9$ and
$v$ is a prime of good reduction of $E$ such that Frobenius at $v$ is an element of order 7 in $G$, then the Euler factor at $v$ is 
$\frac{1}{(1-\zeta_7 \alpha p^{-s})(1-\zeta_7 \beta p^{-s})(1-\zeta_7^2 \alpha p^{-s})(1-\zeta_7^2 \beta p^{-s})(1-\zeta_7^4 \alpha p^{-s})(1-\zeta_7^4 \beta p^{-s})}$, which is visibly not a cube; here $\alpha$ and $\beta$ are the Frobenius eigenvalues at $v$ of $E$, and $\zeta_7$ a suitable primitive 7-th root of unity.
\end{remark}

\begin{question}
For a self-dual Artin representation $\tau$, the sign in the functional equation of $L(E,\tau,s)$ determines the parity of the order of vanishing at $s=1$. The normalised $L$-function $\Lambda(E,\tau,s)$ has the ``clean'' functional equation $\Lambda(E,\tau,s)=\pm\Lambda(E,\tau,2-s)$, so, in particular, the Taylor series expansion around $s\!=\!1$ has either only even terms or only odd terms. Is there any such effect for the $L$-functions in Theorem \ref{mainthm}, i.e. can one normalise them so that the only non-zero coefficients in the Taylor expansion $\Lambda(E,\tau,s)=\sum a_k(s-1)^k$ are the $a_k$ with $p|k$?
\end{question}


\bigskip

\section{Birch--Swinnerton-Dyer conjecture and the Schur index}

Statements that concern the Birch--Swinnerton-Dyer conjecture usually suppose properties about a given $L$-function so as to ascertain information about the rank (e.g. Coates--Wiles, Gross--Zagier, Kolyvagin). Our approach is somewhat peculiar: we are traversing the opposite direction by using the Mordell--Weil group to derive a feature of the $L$-function. We rely on the following generalisation of the Birch--Swinnerton-Dyer conjecture.

\begin{conj}[Birch--Swinnerton-Dyer, Deligne--Gross; see \cite{Rohrlich1990} p.127]
\label{BSDDG}
Let $A$ be an abelian variety over a number field $K$, and let $\tau$ be a representation of $\Gal(F/K)$ for some finite Galois extension $F/K$. Then $L(A/K,\tau,s)$ has analytic continuation to $\C$ and $$\ord_{s=1} L(A/K,\tau,s) = \langle A(F)_{\C}, \tau \rangle,$$ where $A(F)_{\C}$ is the natural representation of $\Gal(F/K)$ on $A(F) \otimes_{\zz} \C$.
\end{conj}

The key observation is that since the Galois group acts on a $\zz$-lattice, $A(F)_{\C}$ is a rational representation. Therefore certain complex irreducible representations $\tau$ cannot appear with multiplicity 1 in $A(F)_\C$; this aspect is measured by the Schur index $m_{\Q}(\tau)$. In contrast, the analogous property is not obvious (and unknown in general) for either the $L$-function of an abelian variety or the $\Q_\ell$-representation on the dual Selmer group $X_{\ell}(A/F)$.

\begin{definition}
Let $G$ be a finite group and $\mathcal{F}$ a subfield of $\C$.
We say a complex representation $\tau$ of $G$ is realisable over $\cF$ if it is conjugate to a representation that factors as $G\to \GL_n(\cF)\subset \GL_n(\C)$ for some $n$.
The Schur index $m_{\mathcal{F}}(\tau)$ is the maximal integer $m$ such that for all representations $\sigma$ of $G$ that are realisable over $\mathcal{F}$, the multiplicity $\langle  \tau,\sigma \rangle$ is a multiple of $m$.
\end{definition}

\begin{example}
\label{q8}
The Schur index $m_{\Q}(\tau)$ of the 2-dimensional irreducible representation $\tau$  of the quaternion group $Q_8$ is 2. Hence $\tau$, despite having rational trace, cannot be realised by matrices in $\GL_2(\Q)$; however $\tau \oplus \tau$ is realisable in $\GL_4(\Q)$.
\end{example}

\begin{remark}\label{schurbound}
Note that for any field $\cF$, $m_{\mathcal{F}}(\tau) \leq \dim \tau$ as the regular representation is realisable over $\qq$. In fact $m_{\mathcal{F}}(\tau)$ always divides the dimension $\dim \tau$, see e.g. \cite[Corollary 10.2]{Isaacs1976}.
\end{remark}

\begin{theorem}
\label{mainschurthm}
Let $F/K$ be a Galois extension of number fields, and let $\tau$ be an irreducible Artin representation of $\Gal(F/K)$. Then for all abelian varieties $A/K$,
the multiplicity of $\tau$ in $A(F)_\C$ is divisible by $m_\Q(\tau)$. In addition:
\begin{enumerate}
  \item If Conjecture \ref{BSDDG} holds, then $\ord_{s=1} L(A/K,\tau,s)$ is divisible by $m_{\qq}(\tau)$;
  \item If $\sha(A/F)[\ell^\infty]$ is finite for some prime $\ell$, then $\langle X_{\ell}(A/F), \tau\rangle$ is divisible by $m_{\qq}(\tau)$.
\end{enumerate}
\end{theorem}

\begin{proof}
By construction, $A(F)_{\C}$ is realisable over $\Q$ so by definition $m_{\qq}(\tau)$ divides $\langle A(F)_{\C}, \tau \rangle$.
The $L$-function statement now follows directly from Conjecture~\ref{BSDDG}.
If $\sha(A/F)[\ell^\infty]$ is finite, then $X_{\ell}(A/F)\cong A(F) \otimes_{\zz} \Q_{\ell}$ as $\Q_{\ell}[\Gal(F/K)]$-modules, from which the second part follows.
\end{proof}

\begin{remark}
Without the finiteness assumption on $\sha$, the dual Selmer group $X_{\ell}(A/F)$ is not known to be a rational or even an orthogonal representation of the Galois group (although it is known to be self-dual, see \cite{Dokchitser2009}).
Thus, as the $\ell$-adic Schur index $m_{\Q_{\ell}}(\tau)$ can be 1, there is no obvious representation-theoretic reason for the multiplicity of $\tau$ in $X_{\ell}(A/F)$ to be a multiple of $m_{\Q}(\tau)$; see Theorem \ref{schurindices} for an example of such a $\tau$.
\end{remark}

\begin{remark}
\label{brauerspeiser}

The reason for the restriction on the order of vanishing of the $L$-function is fairly well-understood for self-dual representations $\tau$ with Schur index 2 (for example the quaternion representation in Example 2.3). In this case the conjectural functional equation is of the form $L(A,\tau,s)= \pm L(A,\tau,2-s)\times(\Gamma\text{-factors and exponential})$. So the parity of the order of vanishing at $s\!=\!1$ is determined by the sign $\pm$, which is given by the global root number $W(A,\tau)$ and known to be + whenever $\tau$ is symplectic and in many cases when $\tau$ is orthogonal with Schur index~2, see \cite[Proposition 2]{Rohrlich1996} and \cite[Theorem 0.1]{Sabitova2007}.

It is tempting to use the sign in the functional equation to force a zero of the \hbox{$L$-function} for a representation $\tau$ with large Schur index $m=m_\Q(\tau)$. If  Conjecture~\ref{BSDDG} is true, the order of vanishing is a fortiori at least $m$. Curiously enough, this is impossible to achieve: if $m>2$, the representation $\tau$ cannot be self-dual by the Brauer--Speiser theorem. Thus the functional equation relates $L(A,\tau,s)$ to $L(A,\tau^*,2-s)$, and the root number cannot be used to force the $L$-function to vanish at $s=1$. 
\end{remark}


\section{Schur indices in $C_q\rtimes C_{p^n}$}

We now compute the Schur indices of representations of $C_q \rtimes C_{p^n}$ appearing in Theorem \ref{mainthm}. We only prove that the Schur index is divisible by $p$ without determining it exactly, so the bounds on orders of vanishing of $L$-functions that we have given may be suboptimal. For example, if $\tau$ is an irreducible faithful representation of $C_{19} \rtimes C_{3^4}$ (with the largest possible action), then $m_{\Q}(\tau)\!=\!9$.

For a field $\cF$ and representation $\tau$, we let $\cF(\tau)$ denote the finite abelian extension of $\cF$ 
generated by the values of the trace of $\tau$. We further let $\zeta_m$ denote a primitive $m^{th}$ root of unity and $N_{\cF/\cK}$ be the norm map for any field extension $\cF/\cK$.

\pagebreak

\begin{proposition}
\label{reps}
Let $p,q$ be distinct odd primes and $G = C_q \rtimes C_{p^n}$, where the image of $C_{p^n}$ in $\Aut C_q$ has order $p^r$.
Let $\tau$ be a complex irreducible representation of $G$. Write $X=C_q\times C_{p^{n-r}}\lhd G$.

\noindent (i) If $\tau$ is unfaithful then $\tau$ is lifted either from $C_{p^n}$ or from $C_q \rtimes C_{p^{n-1}}$. 

\noindent (ii) If $\tau$ is faithful, then $\dim\tau\!=\!p^r$ and there is a faithful 1-dimensional representation of $X$ such that $\tau\!=\!\Ind_X^G \psi$. Conversely, the induction of a faithful 1-dimensional representation $\psi$ of $X$ gives a faithful irreducible representation of $G$. 

\noindent (iii) Every faithful irreducible representation $\tau$ of $G$ may be written as $\tau_r \otimes \chi$ for some faithful irreducible representation $\tau_r$ of $C_q \rtimes C_{p^r}$ and faithful 1-dimensional representation $\chi$ of $C_{p^n}$.

\noindent (iv) If $\tau=\Ind_X^G\psi$ is faithful and $\mathcal{F}\!\subset\!\C$ is a field, then $\cF(\psi)=\cF(\zeta_{p^{n-r}},\zeta_q)$ and $\cF(\tau)=\cF(\zeta_{p^{n-r}},\sum_{t\in H}\zeta_q^t)$, where $H \le (\Z / q\Z)^{\times}$ is the subgroup of order $p^r$.

\noindent (v) If $\tau=\Ind_X^G\psi$ is faithful and $\mathcal{F}\!\subset\!\C$ is a field such that $[\cF(\psi)\!:\!\cF(\tau)]=p^r$, then the Schur index $m_{\mathcal{F}}(\tau)\!=\!1$ if and only if $\zeta_{p^{n-r}}$ is  in the image of $N_{\mathcal{F}(\psi)/\mathcal{F}(\tau)}$.
\end{proposition}

\begin{proof}
The group $G$ has presentation $G=\langle a,b\>|\>a^q=b^{p^n}\!=\text{id}, bab^{-1}\!=a^j \rangle$ where $j$ has order $p^r$ modulo $q$. The subgroup $X$ is $\langle a,b^{p^r}\rangle$; it is the centraliser of~$C_q$. For a representation $\psi$ of $X$ and a element $g\in G$ we write ${}^g\psi$ for the conjugate representation defined by ${}^g\psi(h)=\psi(ghg^{-1})$.

(i)
The maximal quotients of $G$ are $C_{p^n}$ and (if $r<n$) $C_q \rtimes C_{p^{n-1}}$, so if $\tau$ is not faithful, it factors through one of these.

(ii)
By \cite[Proposition 25]{Serre1977}, every faithful representation of $G$ is necessarily induced from a 1-dimensional representation $\psi$ of $X$; in particular $\dim \tau=p^r$. 
Moreover, since $\ker\psi$ is normal in $G$ (as $X$ is normal in $G$ and $\ker\psi$ is characteristic in the cyclic group $X$), we have $\ker \psi\subseteq \ker\tau$, and hence $\psi$ must be faithful.

Conversely, $h\mapsto b^{k}h b^{-k}$ are distinct automorphisms of $X$ for $0\!\le\! k\!<\!p^r\!-\!1$, so if $\psi$ is a faithful 1-dimensional representation of $X$, then $\psi, {}^b\psi, \ldots, {}^{b^{p^r-1}}\psi$ are all distinct. Thus $\langle\tau,\tau\rangle=\langle\psi,\Res^G_X\Ind^G_X\psi\rangle=\langle\psi,\bigoplus_{0\le k<p^r} {}^{b^k}\psi\rangle=1$ by Frobenius reciprocity and Mackey's formula, and so $\tau$ is irreducible. It is clearly faithful by~(i).

(iii) Let $\tau=\Ind_X^G \psi$, for some faithful 1-dimensional $\psi$ of order $qp^{n-r}$. We can rewrite this as $\psi = \psi_q \otimes \psi_{p^{n-r}}$ where $\psi_m$ has order $m$. Now $\tau_r=\Ind_X^G \psi_q$ is the inflation of a faithful representation of $C_q \rtimes C_{p^r}$. Let $\chi$ be a 1-dimensional representation of $G$ which factors through $C_{p^n}$ such that $\Res_X^G \chi = \psi_{p^{n-r}}$. The push-pull formula shows that $\tau = \tau_r \otimes \chi$, as claimed. 

(iv) If $\tau$ is faithful, then by (ii) $\psi$ is a faithful 1-dimensional representation of $X \cong C_{qp^{n-r}}$, hence $\cF(\psi)=\cF(\zeta_{qp^{n-r}})$. To compute $\cF(\tau)$, it suffices to compute the induced character on the conjugacy classes of $G$ which have nonempty intersection with $X$. Since $X \triangleleft G$, it follows that $\cF(\tau)=\cF(\Res_X^G \tau)$.

Note that $b^{p^r}$ is central in $G$ and $\tau$ is irreducible so $\tau(b^{p^r})$ must be scalar by Schur's lemma; as $\Res_X^G \tau$ contains $\psi$ as a constituent, this scalar is multiplication-by-$\zeta_{p^{n-r}}$, hence $\zeta_{p^{n-r}} \in \cF(\tau)$. 
For $a^xb^{p^ry} \in X$ we  have $\trace \tau(a^xb^{p^ry})=\zeta_{p^{n-r}}^{y} \trace \tau(a^x)$, so 
$\cF(\tau)$ is generated over $\cF$ by $\zeta_{p^{n-r}}$ and the traces $\trace \tau(a^x))$ for $1\le x\le q$.

As in the proof in (ii), \smash{$\Res^G_X\tau=\bigoplus_{0\le k<p^r} {}^{b^k}\psi$}, so $\trace\tau(a^x)=\sum_{t \in H} \zeta_q^{xt}$, where $H$ is the unique index subgroup of order $p^r$ contained in $(\Z/q\Z)^{\times}$.
Note that for any polynomial $f \in \Q[X]$, $f(\zeta_q)$ is $\Gal(\overline{\Q}/\Q)$-conjugate to $f(\zeta_q^x)$ whenever $q\!\nmid\! x$, and hence $f(\zeta_q^x) \in \Q(f(\zeta_q))$ since $\Q(f(\zeta_q))/\Q$ is abelian. In particular, letting $f(X)=\sum_{t \in H} X^t$ (where we fix representatives for $H$), we see that $\sum_{t \in H} \zeta_q^{xt} \in \Q(\sum_{t \in H} \zeta_q^{t})$ for all $x$. Hence $\cF(\tau)=\cF(\zeta_{p^{n-r}},\sum_{t \in H} \zeta_q^{t})$ as claimed.

(v) First note that $X$ is normal, abelian and equal to its own centraliser, $X = C_G(X)$, as otherwise $b^k \in C_G(X)$ for some $k$ with $p^r \nmid k$ which doesn't commute with $a$. Since by assumption the (abelian) extension $\cF(\psi)/\!\cF(\tau)$ has degree $p^r$, the representation $\psi$ must have $p^r$ distinct $\Gal(\cF(\psi)/\!\cF(\tau))$-conjugates, which then must be precisely the constituents of $\Res^G_X\tau$.
Thus $(G,X,\tau)$ is an $\cF$-triple, in the terminology of \cite[Definition 10.5]{Isaacs1976}. Noting that $G=X C_{p^n}$, it then follows from \cite[Theorem 10.10]{Isaacs1976} that $m_{\mathcal{F}}(\tau)=1$ if and only if $\zeta_{p^{n-r}} \in N_{\mathcal{F}(\psi)/\mathcal{F}(\tau)} \mathcal{F}(\psi)$.
\end{proof}

\begin{theorem}
\label{schurindices}

Let $p,q$ be distinct odd primes and $G = C_q \rtimes C_{p^n}$, where the image of $C_{p^n}$ in $\Aut C_q$ has order $p^r$ and $0<r \leq n$. Let $\tau$ be a complex irreducible faithful representation of $G$. 
Then:\\
(i) The Schur index $m_{\Q}(\tau)=p^s$ for some $0<s\le r$ if $p^n\nmid q\!-\!1$, and is $1$ otherwise; \\
(ii) The Schur index $m_{\Q_{q}}(\tau)=m_{\Q}(\tau)$;\\
(iii) The Schur index $m_{\Q_{\ell}}(\tau)=1$ for every prime $\ell \neq q$.

\end{theorem}

\begin{proof}
(iii) It is a general fact that if $\ell \nmid |G|$, then $m_{\Q_{\ell}}(\tau)=1$; see for example \cite{Gow1975}. The Corollary in \cite{Gow1975} states more generally that if $\tau$ is irreducible modulo $\ell$, then $m_{\Q_{\ell}}(\tau)=1$; this will be our approach for the case $\ell=p$. To see that this holds, let $\sigma$ be an irreducible constituent of $\tau$ modulo $p$. Now the eigenvalues of $\tau(a)$ (using the notation from the first paragraph of the proof of Proposition \ref{reps}) are primitive $q^{\text{th}}$ roots of unity, hence this also holds for $\sigma$. Let $v$ be an eigenvector for $\sigma(a)$ with eigenvalue $\zeta$. Then $\sigma(b^{-1})v$ is also an eigenvector for $\sigma(a)$ with eigenvalue $\zeta^j$. As $j$ has order $p^r$ modulo $q$ (note $q \neq p$), $\sigma$ has $p^r$ distinct eigenvalues, so $\dim\sigma=\dim\tau$ and hence $\tau$ is irreducible modulo $p$.

(ii) The global Schur index $m_{\Q}(\tau)$ is well known to equal the lowest common multiple of the local Schur indices $m_{\Q_v}(\tau)$ for all places $v$ of $\Q$ (see for example \cite[Theorem~2.4]{Olteanu2009}). Now $\tau$ is not self-dual (as $G$ has odd order) so $m_{\mathbb{R}}(\tau)=1$ hence the result is immediate from (iii).

(i) We prove instead the same statement for $m_{\Q_q}(\tau)$; the global statement for $m_{\Q}(\tau)$ then follows from (ii). Write $\tau=\Ind_X^G\psi$, as in Proposition \ref{reps}(ii). By Proposition \ref{reps}(iv), the extension $\Q_q(\psi)/\Q_q(\tau)$ is totally ramified of degree $p^r$, and so by (v) it suffices to check whether $\zeta_{p^{n-r}}$ is in the image of the norm map $N_{\Q_q(\psi)/\Q_q(\tau)}$.

By local class field theory, the subgroup of $\cO_{\Q_q(\tau)}^\times$ consisting of norms from $\cO_{\Q_q(\psi)}^\times$ has index $p^r$. Furthermore, as the extension is tame, $u\in\cO_{\Q_q(\tau)}^\times$ is a norm if and only if its image $\bar{u}$ in the residue field $\F_{\Q_q(\tau)}$ of $\Q_q(\tau)$ is a norm from  the residue field of $\Q_q(\psi)$; as the two residue fields are the same, this is equivalent to $\bar{u}$ being of the form $\bar{u}=x^{p^l}$ for some $x\in\F_{\Q_q(\tau)}$.

Thus we are reduced to checking whether $\F_{\Q_q(\tau)}$ contains a primitive $p^n$-th root of unity. Since $\Q_q(\tau)/\Q_q(\zeta_{p^{n-r}})$ is totally ramified (Proposition \ref{reps}(iv)), by Hensel's Lemma this happens if and only if $\zeta_{p^n}\in\Q_q(\zeta_{p^{n-r}})$.

If $p^n|q\!-\!1$, then $\zeta_{p^n}\in\Q_q\subseteq \Q_q(\zeta_{p^{n-r}})$, and hence $m_{\Q_q}(\tau)=1$.

Conversely, if $p^n\nmid q\!-\!1$, then $q\bmod p^n$ is a non-trivial element of $p$-power order (since $r>0$ implies $q\equiv 1\bmod p$) in $(\Z/p^n\Z)^\times$. In particular, $\Gal(\Q_q(\zeta_{p^n})/\Q_q)$ contains an element of order $p$. All such elements fix $\zeta_{p^{n-r}}$, and consequently $\Q_q(\zeta_{p^{n-r}})\neq \Q_q(\zeta_{p^n})$. It follows that $\zeta_{p^n}\not\in\Q_q(\zeta_{p^{n-r}})$ and so $m_{\Q_q}(\tau)\neq 1$. It now follows from Remark \ref{schurbound} and Proposition~\ref{reps}(ii) that the Schur index is $m_{\Q_q}(\tau)=p^s$ for some $0<s \leq r$.
\end{proof}

\begin{ack}
The authors would like to thank Tim Dokchitser for the useful discussions that led to this paper. The second author is supported by a Royal Society University Research Fellowship.
\end{ack}

\begin{data}
All relevant results are contained within the article; no supporting data is held elsewhere.
\end{data}

\bibliographystyle{alpha}
\bibliography{schur}

\end{document}